\documentclass{article}

\usepackage[utf8]{inputenc}
\usepackage[T1]{fontenc}
\usepackage[english,french]{babel}
\usepackage{lmodern}
\usepackage{calrsfs}
\usepackage{color}
\usepackage{amsmath}
\usepackage{amsthm}
\usepackage{mathtools, bm}
\usepackage{amssymb, bm}
\usepackage{sectsty}
\usepackage{graphicx}
\usepackage{float}
\usepackage{cancel}
\usepackage{url}
\usepackage{nomencl}
\makenomenclature
\usepackage{makeidx}
\makeindex

\usepackage{here}
\usepackage{geometry}
\usepackage[pdftex]{hyperref}
\usepackage{pifont}
\usepackage{dsfont}

\usepackage{tikz}
\usetikzlibrary{matrix,arrows}
\usetikzlibrary{shapes}
\usetikzlibrary{calc}
\usetikzlibrary{arrows}
\usetikzlibrary{decorations.pathreplacing,decorations.markings}

\chapterfont{\centering}
\sectionfont{\centering}
\subsectionfont{}

\newtheoremstyle{monstyledem} 
    {8pt}                    
    {8pt}                    
    {\normalfont}                   
    {}                           
    {\bf}                   
    {\newline}                          
    {.5em}                       
    {}  

\newtheoremstyle{monstyle} 
    {8pt}                    
    {8pt}                    
    {\itshape}                   
    {}                           
    {\bf}                   
    {\newline}                          
    {.5em}                       
    {}  

\theoremstyle{monstyle}

\newtheorem{thm}{Théorème}[section]
\newtheorem{de}{Définition}[section]
\newtheorem{prop}{Proposition}[section]
\newtheorem{lem}{Lemme}[section]
\newtheorem{cor}{Corollaire}[section]
\newtheorem{rem}{Remarque}[section]

\theoremstyle{monstyledem}

\newcommand\blankfootnote[1]{%
  \let\thefootnote\relax\footnotetext{#1}%
  \let\thefootnote\svthefootnote%
}

\title{Méthode de Mahler en caractéristique non nulle : un analogue du Théorème de Ku. Nishioka}
\author{Gwladys Fernandes }
\newcommand{\address}{{
\bigskip
\noindent
Gwladys Fernandes\\{\small  Université Claude Bernard Lyon 1,}\\ {\small 43 boulevard du 11 novembre 1918,} \\ {\small 69622 Villeurbanne Cedex, France}\\
\textit{Courriel} : {\small fernandes@math.univ-lyon1.fr}}}

\begin{document}
\maketitle

\blankfootnote{\begin{footnotesize}
This project has received funding from the European Research Council (ERC) under the European Union's Horizon 2020 research and innovation programme
under the Grant Agreement No 648132.
\end{footnotesize}}

\noindent
\textbf{{\footnotesize Résumé}} : {\footnotesize En 1990, Ku. Nishioka \cite{N-art} démontre un théorème fondamental pour la méthode de Mahler, qui constitue l'analogue du théorème de Siegel-Shidlovskii pour les fonctions mahlériennes. Le but de cet article est d'établir une version du théorème de Ku. Nishioka qui soit également valable pour des systèmes mahlériens définis sur des corps de fonctions en caractéristique non nulle. Nous reprenons l'approche introduite dans un cas particulier par Denis en 1999. Celle-ci s'appuie sur un critère d'indépendance algébrique général dû à Philippon. La motivation principale de notre travail repose sur le fait remarquable, découvert par Denis, que dans le contexte des corps de fonctions en caractéristique non nulle, des analogues de périodes comme $\pi$ ou les valeurs aux entiers de la fonction $\zeta$ de Riemann s'obtiennent comme valeurs de fonctions mahlériennes en des points algébriques.}

\section{Introduction}

La problématique à l'origine de ce texte peut être énoncée ainsi : y a t-il équivalence entre l'indépendance algébrique sur $\overline{\mathbb{Q}}(z)$ de fonctions analytiques $f_{1}(z),\cdots, f_{n}(z)$ dont les coefficients de Taylor sont dans l'ensemble $\overline{\mathbb{Q}}$ des nombres algébriques, et l'indépendance algébrique sur $\overline{\mathbb{Q}}$ de leurs valeurs $f_{1}(\alpha),\cdots, f_{n}(\alpha)$, en un point algébrique $\alpha$ ? 
\medskip

Une relation algébrique non triviale sur $\overline{\mathbb{Q}}(z)$ entre les fonctions fournit par évaluation en $z=\alpha$ une relation algébrique non triviale sur $\overline{\mathbb{Q}}$ entre leurs valeurs en $\alpha$ : l'indépendance algébrique des valeurs entraîne donc celle des fonctions. La réciproque est fausse en général (voir par exemple la discussion dans \cite[p.250]{M4}). Toutefois, en demandant aux fonctions de vérifier certaines équations fonctionnelles et en sélectionnant des nombres algébriques adaptés, il existe au moins trois cadres dans lesquels sa véracité a été établie : le théorème de Siegel-Shidlovskii \cite{S-S} pour des solutions de systèmes différentiels, le théorème de Ku. Nishioka \cite{N-art} pour des solutions de systèmes mahlériens, et le théorème de Anderson, Brownawell et Papanikolas \cite{ABP1} (en caractéristique non nulle) pour des solutions de systèmes aux $\sigma$-différences.
\medskip

Le théorème de Ku. Nishioka est un théorème fondamental pour la méthode de Mahler. L'objectif de cet article est d'en établir une version qui soit également valable pour des systèmes mahlériens définis sur des corps de fonctions en caractéristique non nulle. Nous commençons par rappeler brièvement le double contexte (détaillé dans la section \ref{contextes}) dans lequel nous nous placerons, en soulignant la forte analogie existant entre corps de nombres et corps de fonctions en caractéristique non nulle. Cela nous permettra d'énoncer notre résultat principal avec des notations communes aux deux contextes en présence.

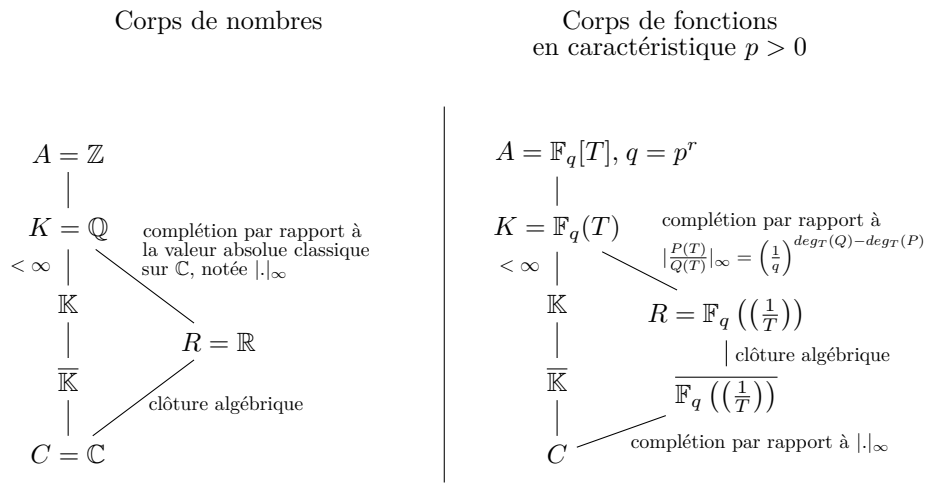
\begin{figure}[H]
\begin{center}
\begin{tikzpicture} 

\node(Z) at (1,4) {$A=\mathbb{Z}$}; 
\node(Q) at (1,3) {$K=\mathbb{Q}$}; 
\node(Kd) at (1,2) {$\mathbb{K}$};
\node(Kdb) at (1,1) {$\overline{\mathbb{K}}$};
\node(Cd) at (1,0) {$C=\mathbb{C}$};

\node(Rd) at (3,1.5) {$R=\mathbb{R}$};

\node(car0) at (3,5.75) {Corps de nombres};
\node(carp) at (9,5.75) {Corps de fonctions};
\node(carp') at (9,5.4) {en caractéristique $p>0$};

\node(trait0) at (6,4.75) { }; 
\node(trait1) at (6,-0.5) { }; 

\node(A) at (7.5,4) {$A=\mathbb{F}_{q}[T]$,};
\node(B) at (8.9,4) {$q=p^{r}$};
\node(K) at (7.5,3) {$K=\mathbb{F}_{q}(T)$}; 
\node(Kd') at (7.5,2) {$\mathbb{K}$};
\node(Kdb') at (7.5,1) {$\overline{\mathbb{K}}$};
\node(C) at (7.5,0) {$C$};

\node(R) at (9.75,1.85) {$R=\mathbb{F}_{q}\left(\left(\frac{1}{T}\right)\right)$};
\node(Rb) at (9.75,0.8) {$\overline{\mathbb{F}_{q}\left(\left(\frac{1}{T}\right)\right)}$};

\draw[-] (Z) -- (Q);
\draw[-] (Q) -- (Kd);
\draw[-] (Kd) -- (Kdb);
\draw[-] (Kdb) -- (Cd);
\draw[-] (Q) -- (Rd);
\draw[-] (Rd) -- (Cd);

\draw[-] (trait0) -- (trait1);

\draw[-] (A) -- (K);
\draw[-] (K) -- (Kd');
\draw[-] (Kd') -- (Kdb');
\draw[-] (Kdb') -- (C);
\draw[-] (K) -- (R);
\draw[-] (R) -- (Rb);
\draw[-] (Rb) -- (C);

\node(c1)[scale=0.75] at (3.45,2.95) {complétion par rapport à};
\node(c2)[scale=0.75] at (3.5,2.7) {la valeur absolue classique};
\node(c3)[scale=0.75] at (2.95,2.45) {sur $\mathbb{C}$, notée $|.|_{\infty}$};
\node(ca)[scale=0.75] at (3.10,0.65) {clôture algébrique};
\node(cp1)[scale=0.75] at (10.35,3.1) {complétion par rapport à};
\node(cp2)[scale=0.75] at (10.65,2.65) {$|\frac{P(T)}{Q(T)}|_{\infty}=\left(\frac{1}{q}\right)^{deg_{T}(Q)-deg_{T}(P)}$};
\node(cap)[scale=0.75] at (10.2,0.15) {complétion par rapport à $|.|_{\infty}$};
\node(f)[scale=0.75] at (0.5,2.5) {$<\infty$};
\node(fp)[scale=0.75] at (7,2.5) {$<\infty$};
\node(cap2)[scale=0.75] at (10.9,1.3) {clôture algébrique};

\end{tikzpicture}
\end{center}
\end{figure}

Étant donné un entier naturel $d\geq 2$, une fonction $f(z)\in \mathbb{K}[[z]]$ est dite $d$-malhérienne sur $\mathbb{K}(z)$ s'il existe des polynômes $P_{0}(z),\cdots P_{m}(z)\in\mathbb{K}[z]$, $P_{m}(z)$ $\cancel\equiv$  $0$, tels que 
$$P_{0}(z)f(z)+P_{1}(z)f(z^{d})+\cdots+P_{m}(z)f(z^{d^{m}})=0.$$ 
L'entier $m$ est appelé l'ordre de $f(z)$.
On dit que le vecteur colonne composé des fonctions\\ $f_{1}(z),\cdots, f_{n}(z)\in\mathbb{K}[[z]]$ vérifie un système $d$-mahlérien s'il existe une matrice $A(z) \in GL_{n}(\mathbb{K}(z))$ telle que

\begin{equation}
\label{syst_gen}
\begin{pmatrix}
f_{1}(z^{d}) \\ \vdots \\ f_{n}(z^{d}) 
\end{pmatrix} =A(z)\begin{pmatrix}
f_{1}(z) \\ \vdots \\ f_{n}(z) 
\end{pmatrix}.
\end{equation}

Toute fonction $d$-mahlérienne s'insère dans un système $d$-mahlérien à l'aide d'une matrice compagnon et réciproquement, toute fonction coordonnée d'un vecteur vérifiant un système $d$-mahlérien est $d$-mahlérienne. 
\begin{de}
On dira qu'un nombre $\alpha\in C$ est régulier pour le système \eqref{syst_gen} si pour tout entier naturel $k$, le nombre $\alpha^{d^{k}}$ n'est pas un pôle de la matrice $A(z)$ ni de la matrice $A^{-1}(z)$. 
\end{de}

\begin{rem}
\em{
Si l'on suppose les fonctions $\{f_{i}(z)\}_{1\leq i\leq n}$ analytiques au voisinage de l'origine et $0<|\alpha|_{\infty}<1$, alors, l'hypothèse selon laquelle pour tout $k\in\mathbb{N}$ le nombre $\alpha^{d^{k}}$ n'est pas un pôle de $A^{-1}(z)$ garantit que les fonctions $\{f_{i}(z)\}_{1\leq i\leq n}$ sont bien définies en $\alpha^{d^{k}}$, pour tout $k\in\mathbb{N}$.}
\end{rem}

%
%
\medskip

Notre résultat principal est le suivant. Il est valable pour tout choix de l'un ou l'autre des deux contextes du schéma précédent.

\begin{thm}
\label{nish_gen}
Soit $\mathbb{K}$ une extension finie de K. Soient $n\geq 1$, $d\geq 2$ deux entiers et $f_{1}(z), \ldots, f_{n}(z) \in \mathbb{K}\{z\}$ des fonctions analytiques au voisinage de l'origine vérifiant le système $d$-mahlérien \eqref{syst_gen}.
%
Soit $\alpha\in\overline{\mathbb{K}}$, $0<|\alpha|_{\infty}<1$, un nombre régulier pour le système \eqref{syst_gen}.\\
Alors, l'égalité suivante est vérifiée :
\begin{equation}
\label{egdegtr_gen}
degtr_{\mathbb{K}}\{f_{1}(\alpha), \ldots, f_{n}(\alpha)\}=degtr_{\mathbb{K}(z)}\{f_{1}(z), \ldots, f_{n}(z)\}.
\end{equation}
 \end{thm}

Le cas où $\mathbb{K}$ est un corps de nombres est dû à Ku. Nishioka. Il s'agit de l'analogue du théorème de Siegel-Shidlovskii pour les fonctions mahlériennes, que Kubota \cite{Kubota,Kubota2} et Loxton et van der Poorten \cite{LvdPI,LvdPII,LvdPIII} entre autres cherchaient à obtenir depuis les années 1970. La démonstration établie par Ku. Nishioka \cite{N-art} (voir aussi \cite{N}) s'appuie sur la méthode de Mahler et la théorie de l'élimination mise en place par Nesterenko (voir par exemple \cite{Nest}) à la fin des années 1970. 

Nous reprenons l'approche introduite par Denis \cite{D} en 1999 dans sa preuve d'un cas particulier du théorème \ref{nish_gen}. Celle-ci se fonde sur un critère d'indépendance algébrique établi par Philippon \cite{Ph2} (voir aussi \cite{Ph1}). Cette méthode de démonstration a l'avantage d'être valable aussi bien pour les corps de nombres que pour les corps de fonctions en caractéristique non nulle, et permet notamment de retrouver le théorème de Ku. Nishioka. Notons que ce point de vue ne peut être considéré comme étant radicalement différent de celui adopté par Ku. Nishioka puisque la démonstration du critère d'indépendance algébrique de Philippon contient les outils développés par Nesterenko sur lesquels reposent la démonstration de Ku. Nishioka. Toutefois, sa puissance réside en la validité du critère d'indépendance algébrique de Philippon dans un cadre très général qui confère à la méthode de Mahler, et donc à notre démonstration, son indépendance vis-à-vis de la caractéristique.

\medskip


La motivation principale de notre travail repose sur le fait remarquable, découvert par Denis, que dans le contexte des corps de fonctions en caractéristique non nulle, des analogues de périodes comme $\pi$ ou les valeurs aux entiers de la fonction $\zeta$ de Riemann s'obtiennent comme valeurs de fonctions mahlériennes en des points algébriques. Cela fait de la méthode de Mahler un outil puissant pour l'étude des périodes de modules de Drinfeld (ou plus généralement de t-motifs). L'exemple suivant \cite{D2} servira d'illustration à ce phénomène. Par analogie avec la fonction $\zeta$ de Riemann, on définit pour tout $s\in\mathbb{N}^{*}$ :
$$\zeta_{C}(s)=\sum_{a\in \mathbb{F}_{q}[T],\text{ } a\text{ unitaire}} \frac{1}{a^{s}}.$$


Carlitz a montré dans \cite{C1} l'égalité suivante.
$$\zeta_{C}(s)=\sum_{h=0}^{+\infty} \frac{(-1)^{hs}}{(L_{h})^{s}}, \text{ }\forall 1\leq s\leq p-1,$$
où $L_{0}=1$ et pour tout $h\geq 1$ :
$$L_{h}= \left(T^{q^{h}}-T\right)\left(T^{q^{h-1}}-T\right)\cdots\left(T^{q}-T\right)=\left(T^{q^{h}}-T\right)L_{h-1}.$$

En posant :
\begin{equation}
\label{fctsappli}
f_{s}(z)=\sum_{h=0}^{+\infty} \frac{(-1)^{hs}}{\left(\left(z^{q^{h}}-T\right)\left(z^{q^{h-1}}-T\right)\cdots\left(z^{q}-T\right)\right)^{s}},
\end{equation}
on a :
$$f_{s}(T)=\zeta_{C}(s), \text{ }\forall 1\leq s\leq p-1,$$

et
$$f_{s}(z^{q})=(-1)^{s}\left(z^{q}-T\right)^{s}f_{s}(z)-(-1)^{s}\left(z^{q}-T\right)^{s}, \text{ }\forall 1\leq s\leq p-1,$$

ce qui donne :
\begin{equation}
\label{systex}
\begin{pmatrix}
1 \\ f_{1}(z^{q}) \\ \vdots \\ f_{p-1}(z^{q})
\end{pmatrix} =   
\begin{pmatrix} 
1 & 0 & \cdots & \cdots & 0\\
z^{q}-T & -(z^{q}-T) & 0 & \cdots & 0 \\
\vdots &  & \ddots & & \\
(-1)^{p}\left(z^{q}-T\right)^{p-1} & & & & (-1)^{p-1}\left(z^{q}-T\right)^{p-1}
 \end{pmatrix} \begin{pmatrix}
 1 \\ f_{1}(z) \\ \vdots \\ f_{p-1}(z)
 \end{pmatrix}.
\end{equation}
 
Si les fonctions $\{f_{s}(z)\}_{1\leq s\leq p-1}$ sont algébriquement indépendantes sur $K(z)$, le théorème \ref{nish_gen} implique que les nombres $\{\zeta_{C}(s)\}_{1\leq s\leq p-1}$ le sont sur $K$. Dans cet esprit, Denis a déjà démontré en 2006 que ces nombres sont algébriquement indépendants sur $K$. Celui-ci a en effet prouvé dans \cite{D2} que les $p-1$ fonctions $f_{1}(z),\cdots, f_{p-1}(z)$ sont algébriquement indépendantes sur $K(z)$ et établi un cas particulier du théorème \ref{nish_gen} correspondant aux matrices de la forme \eqref{systex}. C'est en l'appliquant qu'il a obtenu l'indépendance algébrique des $p-1$ nombres $\zeta_{C}(1),\cdots, \zeta_{C}(p-1)$ sur $K$. Il est en fait possible de pousser cette approche afin d'obtenir toutes les relations d'indépendance algébrique entre les valeurs aux entiers de la fonction $\zeta_{C}$ (voir \cite[p.30]{P1}). A titre de comparaison, on conjecture que les nombres $\pi, \zeta(3), \zeta(5), \zeta(7), \cdots$ sont algébriquement indépendants sur $\overline{\mathbb{Q}}$ mais on ne sait toujours pas démontrer que le nombre $\zeta(3)$ est transcendant.

\begin{rem}
\label{remnishgen}
\em{
Notons que l'utilisation de cette méthode, comme celle de la méthode aux $\sigma-$différences évoquée précédemment, nous ramène au calcul du degré de transcendance de fonctions et donc souvent à l'étude d'un groupe de Galois, qu'il n'est pas toujours aisé de réaliser. Cependant, dans certains cas, typiquement lorsque les fonctions vérifient une équation mahlérienne inhomogène d'ordre 1, il est possible de se passer de la théorie de Galois. C'est ce que fait Denis pour démontrer l'indépendance algébrique des fonctions $\{f_{s}(z)\}_{1\leq s\leq p-1}$. 
}  
\end{rem}

Ce texte se découpe comme suit. Dans la section \ref{contextes} nous décrivons les contextes des corps de nombres et des corps de fonctions en caractéristique non nulle dans lesquels nous nous placerons. Dans la section \ref{phil} nous énonçons le critère d'indépendance algébrique de Philippon et dans la section \ref{dem} nous démontrons le théorème \ref{nish_gen}. Enfin, dans la section \ref{mahler} nous établissons et illustrons par un exemple l'analogue d'un théorème de Mahler \cite{M1} dans le cadre des corps de fonctions en caractéristique non nulle, qui constitue le pendant du théorème \ref{nish_gen} pour des fonctions $f(z)\in\mathbb{K}[[z]]$ solutions d'un autre type d'équation mahlérienne, au sens suivant. Il existe un entier $d\geq 2$ et deux polynômes $A(z,X), B(z,X)\in \mathbb{K}[z, X]$, $B(z,X)$ $\cancel\equiv$  $0$, tels que :
\begin{equation}
f(z^{d})=\frac{A(z,f(z))}{B(z,f(z))}.
\end{equation}

\section{Contextes}
\label{contextes}
Dans cette section nous introduisons des notations communes au cadre des corps de nombres et des corps de fonctions en caractéristique non nulle, fondées sur les analogies en présence, et dont le sens dépendra du contexte choisi.

\subsection{{\small Corps de nombres}}

On note $A=\mathbb{Z}$ l'anneau des entiers relatifs, $K=\mathbb{Q}$ le corps des fractions de A. On note $R=\mathbb{R}$ le complété de $K$ pour la valeur absolue usuelle $|.|_{\infty}$. On note $\overline{K}=\overline{\mathbb{Q}}$ et $C=\mathbb{C}$ les clôtures algébriques respectives de $K$ et $R$. Le corps $C$ est complet pour $|.|_{\infty}.$ La notation $\mathbb{K}$ désignera une extension algébrique finie de $K$.

\subsection{{\small Corps de fonctions en caractéristique p>0}}

On fixe un nombre premier $p$ et $q=p^{r}$ et on note $A=\mathbb{F}_{q}[T]$ l'anneau des polynômes en la variable $T$ et à coefficients dans le corps fini $\mathbb{F}_{q}$ de caractéristique $p$, $K=\mathbb{F}_{q}(T)$ le corps des fractions de A, $R=\mathbb{F}_{q}\left(\left(\frac{1}{T}\right)\right)$ le complété de K pour la valeur absolue $|.|_{\infty}$ associée à la valuation (1/$T$)-adique $v_{T}$ définie sur K de la façon suivante :
\begin{align}
v_{T}: K &\longrightarrow K \\
\frac{P(T)}{Q(T)} &\longmapsto deg_{T}(Q) - deg_{T}(P),
\end{align}
où $deg_{T}(P)$ désignera dans toute la suite et pour toute caractéristique le degré du polynôme $P$ en $T$.
Puis :

\begin{align}
|.|_{\infty}: K &\longrightarrow K \\
U &\longmapsto q^{-v_{T}(U)}.
\end{align}

On note $\overline{K}$ et $\overline{R}$ les clôtures algébriques respectives de $K$ et $R$. La valeur absolue $|.|_{\infty}$ se prolonge de façon unique sur $\overline{R}$ et on note $C$ le complété de $\overline{R}$ par rapport à $|.|_{\infty}$. Le corps $C$ est algébriquement clos. On pose par convention : $deg(0)= -\infty$. La notation $\mathbb{K}$ désignera une extension algébrique finie de $K$. 

\subsection{{\small Extensions finies et formule du produit}}
Dans toute la suite, les notations suivantes :
$$A, K, \overline{K}, R, \overline{R}, |.|_{\infty}, C, \mathbb{K},$$
introduites ci-dessus, auront des significations différentes selon que l'on considère des corps de nombres ou des corps de fonctions en caractéristique non nulle. Tout énoncé les contenant sera valable dans chacun de ces deux contextes.
Si $\mathbb{K}$ est une extension finie de K, les symboles $\sum_{w}$ et $\prod_{w}$ désigneront respectivement la somme et le produit sur toutes les places de $\mathbb{K}$. Pour une place $w$ de $\mathbb{K}$ étendant une place $v$ de K, on note $\mathbb{K}_{w}$ et $K_{v}$ les complétés respectifs de $\mathbb{K}$ et K par rapport à $w$ et $v$ et on pose $d_{w}=[\mathbb{K}_{w}:{K_{v}}].$ La notation $|.|$ désignera dans toute la suite la valeur absolue $|.|_{\infty}$ distinguée précédemment.

Pour tout $x\in\mathbb{K}^{*}$, la formule suivante, dite formule du produit, est vérifiée :
\begin{equation}
\label{formprod}
\prod_{w} |x|_{w}^{d_{w}} = 1,
\end{equation}
avec $|x|_{w}=1$ pour toute place $w$ de $\mathbb{K}$ à l'exception d'un nombre fini d'entre elles.

\begin{de}
\begin{enumerate}
\item
Soit $\mathbb{K}$ une extension finie de K et soit $a\in \mathbb{K}.$ On définit la hauteur logarithmique absolue de Weil de $a$ par :
$$h(a)=\frac{1}{[\mathbb{K}:K]}\sum_{w} d_{w} log(max\{1,|a|_{w}\}). $$
\item
Soit $P\in \mathbb{K}[X_{0},\ldots, X_{n}].$ On définit sa hauteur logarithmique absolue de Weil par :
$$h(P)=\frac{1}{[\mathbb{K}:K]}\sum_{w} d_{w} log(max_{\underline{\alpha}}\{1,|a_{\underline{\alpha}}|_{w}\}), $$
où le maximum est pris sur l'ensemble des coefficients de $P$.
\end{enumerate}
\end{de}

\begin{rem}
\em{
Par construction, les quantités $h(a)$ et $h(P)$ ne dépendent pas de l'extension finie de $K$ contenant $a$ et les coefficients de $P$ choisie.}
\end{rem}

Nous rappelons les égalités et inégalités fondamentales suivantes (voir par exemple \cite{P1}).

\begin{prop}
\label{liouville}

\medskip

\begin{enumerate}
\item
Soit $\mathbb{K}$ une extension finie de $K$ et soient $a,b\in \mathbb{K}, n\in \mathbb{Z}$ et $P\in \mathbb{K}[X_{0},\ldots, X_{n}]$. On a 
$$h(a+b)\leq h(a)+h(b),$$
et
$$h(ab)\leq h(a)+h(b).$$
\item
$$\text{ Si a }\neq 0, \text{ alors } h(a^{n})=|n|h(a).$$
\item
Inégalité de Liouville : $$\text{ Si a }\neq 0, \text{ alors } log|a|\geq -[\mathbb{K}:K] h(a).$$

\item
$$ \text{ On a } h(P)\leq\sum_{\underline{\alpha}} h(a_{\underline{\alpha}}),$$
où la somme porte sur l'ensemble des coefficients de $P$.
\end{enumerate}

\end{prop}

\section{Critère d'indépendance algébrique de Philippon}
\label{phil}


%

Le théorème \ref{nish_gen} repose sur le critère d'indépendance algébrique suivant (voir \cite{Ph2} et \cite{D}).

\begin{thm}[Philippon]
\label{philpoly}

Soit $(\omega_{0}=1,\omega_{1},\ldots, \omega_{n})\in C^{n+1}.$\\
Soient $c_{1}> 0$ et $s\in\{0,\ldots, n\}$ et soient $\delta(t), \sigma(t), \epsilon(t), \rho(t), t\in \mathbb{N}$ quatre suites croissantes à valeurs réelles supérieures ou égales à 1 et vérifiant les conditions suivantes:

\begin{enumerate}
\label{inegreelspoly}
\item
$\delta(t)\leq \sigma(t)$
\item
$\epsilon(t)\leq\rho(t+1)$
\item
$(\delta(t)+\sigma(t))\longrightarrow_{t\rightarrow+\infty}+\infty$
\item
$\left(\frac{\epsilon(t)}{(\delta(t)+\sigma(t))\delta(t)^{s}}\right)_{t} \text{ est une suite croissante.}$
\item
\begin{equation}
\label{hypVpoly}
\frac{\epsilon(t)^{s+1}}{\delta(t)^{s-1}[\epsilon(t+1)^{s}+\rho(t+1)^{s}]}\geq c_{1}(\delta(t)+\sigma(t))
\end{equation}  
\end{enumerate}

\noindent On suppose que pour tout entier naturel $t$, il existe un polynôme homogène $P_{t}\in \mathbb{K}[X_{0},\ldots, X_{n}]$ satisfaisant aux hypothèses suivantes:
\begin{enumerate}
\setcounter{enumi}{5}
\item
\begin{equation}
\label{hypdhQpoly}
deg(P_{t})\leq \delta(t), h(P_{t})\leq \sigma(t)
\end{equation}
\item
\begin{equation}
\label{hypevalQpoly}
-\rho(t) \leq log|P_{t}(\omega_{0},\omega_{1},\ldots, \omega_{n})|\leq -\epsilon(t)
\end{equation}
\end{enumerate}

Alors, on a :
$$degtr_{\mathbb{K}}\{\omega_{1},\ldots, \omega_{n}\} \geq s.$$

\end{thm}

On en déduit le corollaire suivant.

\begin{cor}
\label{cas}
Soient $(\omega_{0}=1,\omega_{1},\ldots, \omega_{n})\in C^{n+1},$ $s\in\{0,\ldots, n\}$, et $(n(N))_{N\in\mathbb{N}}$ une suite de nombres entiers pour laquelle il existe une constante $c_{2}$ indépendante de $N$ telle que pour tout $N\in\mathbb{N}$ :
\begin{equation}
\label{n(N)}
n(N)\geq c_{2}N^{s+1}.
\end{equation}
Supposons que pour tous $N,k\in\mathbb{N}$, il existe un polynôme homogène $P_{N,k}\in \mathbb{K}[X_{0},\ldots, X_{n}]$ tel que :

\begin{equation}
\label{degHQbiscor}
deg(P_{N,k})\leq c_{3}N, h(P_{N,k})\leq c_{4}d^{k}N,
\end{equation}

\begin{equation}
\label{evalQbiscor}
-c_{5}d^{k}n(N) \leq log|P_{N,k}(\omega_{0},\ldots, \omega_{n})|\leq -c_{6}d^{k}n(N),
\end{equation} 

où les $c_{i}$ sont des constantes strictement positives indépendantes de $N$ et de $k$.
\medskip

Alors :
$$degtr_{\mathbb{K}}\{\omega_{1},\ldots, \omega_{n}\}\geq s.$$
\end{cor}

\begin{proof}

Fixons un entier $N\in\mathbb{N}$. On pose $t=k$ et $\delta(k)=c_{3}N, \sigma(k)=c_{4}d^{k}N, \epsilon(k)=c_{6}d^{k}n(N), \rho(k)= c_{5}d^{k}n(N).$
\medskip

Montrons qu'avec ces notations et les hypothèses du corollaire \ref{cas}, les conditions du théorème \ref{philpoly} sont vérifiées. 
Les conditions 6 et 7 le sont. On note qu'alors $c_{5}\geq c_{6}$, ce qui implique la condition 2. Les conditions 1 et 3 sont également satisfaites. Il reste à montrer que les conditions 4 et 5 sont remplies.

Pour tout $k\in\mathbb{N}$:
$$\frac{\epsilon(k)}{(\delta(k)+\sigma(k))\delta(k)^{s}}=\frac{c_{6}d^{k}n(N)}{(c_{3}N+c_{4}d^{k}N)c_{3}^{s}N^{s}}.$$
Or, on vérifie que la fonction $f$ définie sur $\mathbb{R}^{+}$ par 
$$f(x)=\frac{c_{6}d^{x}n(N)}{(c_{3}N+c_{4}d^{x}N)c_{3}^{s}N^{s}}$$
est croissante sur $\mathbb{R}^{+}.$ Donc la condition 4 est vérifiée.\\
 
Enfin, un calcul donne pour tout $k\in\mathbb{N}$ :
\begin{align}
\frac{\epsilon(k)^{s+1}}{\delta(k)^{s-1}[\epsilon(k+1)^{s}+\rho(k+1)^{s}]} &=c_{7}\frac{d^{k}n(N)}{N^{s-1}} \\
&\geq c_{8}d^{k}N^{2}, \text{ d'après \eqref{n(N)}},
\end{align}
où $c_{7}$ et $c_{8}$ désignent des constantes strictement positives indépendantes de $N$ et de $k$.
Puisque $(\delta(k)+\sigma(k))=o_{N\rightarrow +\infty}(d^{k}N^{2})$ uniformément en $k$, quitte à augmenter $N$, on peut supposer que la condition 5 est satisfaite. On conclut en appliquant le théorème \ref{philpoly}.

\end{proof}

\section{Démonstration du théorème \ref{nish_gen}}
\label{dem}
Dans cette section nous démontrons le théorème \ref{nish_gen} à l'aide du corollaire \ref{cas}. Nous reprenons les notations de ce théorème. Comme $\alpha$ est algébrique sur $\mathbb{K}$, quitte à remplacer $\mathbb{K}$ par $\mathbb{K}(\alpha)$, on peut supposer que $\alpha\in\mathbb{K}$. Notons $l=degtr_{\mathbb{K}}\{f_{1}(\alpha), \ldots, f_{n}(\alpha)\}$ et $l'=degtr_{\mathbb{K}(z)}\{f_{1}(z), \ldots, f_{n}(z)\}.$ Puisque l'évaluation en $z=\alpha$ d'une relation algébrique non triviale à coefficients dans $\mathbb{K}(z)$ entre les fonctions $f_{i}(z)$ fournit une relation algébrique non triviale à coefficients dans $\mathbb{K}$ entre les nombres $f_{i}(\alpha),$ on a :
$$l\leq l'.$$
Ainsi, démontrer le théorème \ref{nish_gen} revient à démontrer que l'on a $l\geq l'.$ Si $l'=0$, cela est clair. Par conséquent, il suffit de prouver que l'on a $l\geq l'$ lorsque $l'\geq 1$. On suppose donc que $l'\geq 1$.

\begin{proof}[Démonstration du théorème \ref{nish_gen}]
\medskip

On notera dans toute la suite $\omega_{0}=1$ et pour tout $i\in\{1,\ldots, n\},$ $\omega_{i}=f_{i}(\alpha)$. On rappelle que l'on suppose que $\alpha\in\mathbb{K}$ et $l'\geq 1$.\\
D'après le corollaire \ref{cas}, il suffit de construire, pour tous $N,k\in\mathbb{N}$, un polynôme $P_{N,k}(X_{0},\cdots X_{n})\in\mathbb{K}[X_{0},\cdots X_{n}]$ tel que son degré, sa hauteur et la petitesse de son évaluation en $X_{i}=\omega_{i}$ vérifient les inégalités \eqref{degHQbiscor} et \eqref{evalQbiscor}, avec pour tout $N\in\mathbb{N}$ :
$$n(N)\geq c_{1}N^{l'+1}.$$

Afin de construire, à $N$ fixé, un polynôme $P_{N,k}(X_{0},\cdots, X_{n})=P_{k}(X_{0},\cdots, X_{n})\in\mathbb{K}[X_{0},\cdots, X_{n}]$ petit en $X_{i}=\omega_{i}$ (au sens de \eqref{evalQbiscor}), on commencera par construire un polynôme $R_{0}(z,X_{0},\cdots, X_{n})\in\mathbb{K}[z,X_{0},\cdots, X_{n}]$ prenant une petite valeur en $z=\alpha$, $X_{i}=\omega_{i}$. Pour ce faire, on cherchera $R_{0}(z,X_{0},\cdots, X_{n})$ tel que la fonction de $z$ : $R_{0}(z,1,f_{1}(z),\cdots, f_{n}(z))$ prenne une petite valeur en $z=0$. Ceci car en itérerant $k$ fois le système \eqref{syst_gen}, pour un nombre $k$ assez grand, on obtiendra, dans l'idée, un polynôme $R_{k}(z,X_{0},\cdots, X_{n})\in\mathbb{K}[z,X_{0},\cdots, X_{n}]$ tel que :

$$R_{k}(z,1,f_{1}(z),\ldots, f_{n}(z))=R_{0}(z^{d^{k}},1,f_{1}(z^{d^{k}}),\ldots, f_{n}(z^{d^{k}})),$$
dont l'évaluation en $z=\alpha, X_{i}=\omega_{i}$ sera de fait petite. On vérifiera alors que le polynôme $P_{k}(X_{0},\cdots, X_{n})$ défini par :
$$P_{k}(X_{0},\cdots, X_{n})=R_{k}(\alpha,X_{0},\cdots, X_{n})\in\mathbb{K}[X_{0},\cdots, X_{n}]$$
prend une petite valeur en $X_{i}=\omega_{i}$ et satisfait aux hypothèses du corollaire \ref{cas}.

\begin{enumerate}
\item
Construction de $R_{0}(z,X_{0},\ldots, X_{n})\in\mathbb{K}[z,X_{0},\cdots, X_{n}]$.
\medskip

Soit $N$ un entier naturel non nul, fixé jusqu'à la fin de cette preuve. Dans la suite de cette démonstration, pour tout $j\in\mathbb{N}$, le terme $c_{j}$ désignera une constante strictement positive ne dépendant que de $\mathbb{K},$ des $f_{i}(z)$ et de $\alpha$ et $c_{j}(N)$ désignera une constante strictement positive ne dépendant que de $\mathbb{K},$ des $f_{i}(z)$, de $\alpha$ et de $N$.

\medskip
Cherchons $R_{0}(z,X_{0},\ldots, X_{n})$ tel que la fonction de $z$: $R_{0}(z,1,f_{1}(z),\ldots, f_{n}(z))$ soit de grand ordre en $z=0$, ce qui garantira sa petitesse recherchée en $z=0$.
Par définition et quitte à renuméroter, on peut supposer que les fonctions $f_{1}(z), \ldots, f_{l'}(z)$ sont algébriquement indépendantes sur $\mathbb{K}(z).$ Il existe alors un polynôme non nul $R(z,X_{1},\ldots, X_{l'})\in  \mathbb{K}[z,X_{1},\ldots, X_{l'}]$ vérifiant les conditions suivantes :
\begin{enumerate}
\item
$deg_{z}(R)\leq N$
\item
$deg_{X}(R)\leq N$
\item
\begin{equation}
\label{ord}
n(N):=ord_{z=0}R(z,f_{1}(z),\ldots, f_{l'}(z))\geq \frac{N^{l'+1}}{l'!}:=c_{1}N^{l'+1}
\end{equation}
\end{enumerate}
En effet, en regardant $R$ comme un polynôme en les $X_{i}$ à coefficients des polynômes $P(z)$ de $\mathbb{K}[z]$, cela revient à résoudre un système de $\frac{N^{l'+1}}{l'!}$ équations en les $(N+1)\begin{pmatrix} l'+N \\ N \end{pmatrix}$ inconnues que sont les coefficients des polynômes $P(z)$, équations à coefficients dans $\mathbb{K}$. Puisque : $\frac{N^{l'+1}}{l'!}<(N+1)\begin{pmatrix} l'+N \\ N \end{pmatrix}$, il en existe bien une solution non triviale à coefficients dans $ \mathbb{K}.$\\ 
Notons :
$$E_{N}(z)=R(z,f_{1}(z),\ldots, f_{l'}(z))=\sum_{j=n(N)}^{+\infty}a_{j}(N)z^{j}.$$
Comme les fonctions $f_{1}(z),\ldots, f_{l'}(z)$ sont algébriquement indépendantes, $E_{N}(z)$ $\cancel\equiv$  $0$.

\item
Construction de $R_{k}(z,X_{0},\ldots, X_{n})\in\mathbb{K}[z,X_{0},\cdots, X_{n}]$.

Soit $R_{0}(z,X_{0},\ldots, X_{n})\in \mathbb{K}[z,X_{0},\ldots, X_{n}]$ le polynôme homogène de degré $N$ en $X_{0},\ldots, X_{n}$ vérifiant:

\begin{equation}
\label{homog}
R_{0}(z,1,X_{1},\ldots, X_{n})=R(z,X_{1},\ldots, X_{l'}).
\end{equation}

On va construire $R_{k}$ par récurrence sur $k$, en itérant l'application : $z\mapsto z^{d}$. En utilisant le système \eqref{syst_gen} vérifié par $\overline{f}(z)=(f_{1}(z),\ldots, f_{n}(z))^{t},$ et en notant $A_{i}(z)$ la $i$-ème ligne de la matrice $A(z)$ et $<.,.>$ le produit scalaire usuel sur $(C(z))^{n}$, on peut écrire :
\begin{equation*}
R_{0}(z^{d},1,f_{1}(z^{d}),\ldots, f_{n}(z^{d})) = R_{0}(z^{d},1,<A_{1}(z),\overline{f}(z)>,\ldots, <A_{n}(z),\overline{f}(z)>).
\end{equation*}

Comme les coefficients des lignes $A_{i}(z)$ sont des fractions rationnelles et que l'on veut manipuler des polynômes, on considère $a(z)$ un polynôme de $ \mathbb{K}[z]$ tel que pour tout $i\in\{1,\cdots, n\}$, $a(z)A_{i}(z)\in \mathbb{K}[z]^{n}$ et dont, pour tout $k\in\mathbb{N}$, $\alpha^{d^{k}}$ n'est pas un zéro (ce qui est possible puisque pour tout $k\in\mathbb{N}$, $\alpha^{d^{k}}$ n'est pas un pôle de $A(z)$).
Notons $M(z)=a(z)A(z)$ et $M_{i}(z)$ la $i$-ème ligne de la matrice $M(z)$.\\
Cela donne :
\begin{multline}
\label{recpoly}
R_{0}(z^{d},a(z),a(z)f_{1}(z^{d}),\ldots, a(z)f_{n}(z^{d})) \\ = R_{0}(z^{d},a(z),<M_{1}(z),\overline{f}(z)>,\ldots, <M_{n}(z),\overline{f}(z)>).
\end{multline}

On définit alors pour tout $k\geq 1$ :
\begin{equation}
\label{defrec}
R_{k}(z,X_{0},X_{1},\ldots, X_{n}) = R_{k-1}(z^{d},a(z)X_{0},<M_{1}(z),\overline{X}>,\ldots, <M_{n}(z),\overline{X}>).
\end{equation}

où 
$$\overline{X}=(X_{1},\cdots, X_{n})^{t}.$$
On remarque que pour tout $k\in\mathbb{N}$, le polynôme $R_{k}$ est homogène de degré $N$ en $X_{0},\ldots, X_{n}$. Donc en particulier :
$$deg_{X}(R_{k})=N, \text{ }\forall k\in\mathbb{N}.$$

De plus, on obtient par récurrence sur $k$ :

\begin{equation}
\label{egrecbase}
R_{k}(z,1,f_{1}(z),\ldots, f_{n}(z)) = \left(\prod_{j=0}^{k-1}a(z^{d^{j}})\right)^{N}E_{N}(z^{d^{k}}), \text{ }\forall k\in\mathbb{N}.
\end{equation}

Par ailleurs, un calcul rapide montre que pour tout $k$ assez grand par rapport à $N$, disons $k\geq c_{0}(N)$, on a :

\begin{equation}
\label{ENdiff0}
E_{N}(\alpha^{d^{k}}) \neq 0,
\end{equation}

et

\begin{equation}
\label{inegE}
-c_{7}d^{k}n(N)\leq log|E_{N}(\alpha^{d^{k}})|\leq -c_{8}d^{k}n(N).
\end{equation}

L'assertion \eqref{ENdiff0} et le fait que pour tout $k\in\mathbb{{N}}$, $\alpha^{d^{k}}$ n'est pas un zéro de $a(z)$ impliquent que :

\begin{equation}
R_{k}(\alpha, \omega_{0},\ldots, \omega_{n})\neq 0, \text{ } \forall k\geq c_{0}(N).
\end{equation}

Puis, en posant :
$$a(z)=z^{\nu}b(z), \text{ avec } \nu\in\mathbb{N}, \text{ } b(0)\neq 0,$$

on obtient à l'aide de \eqref{egrecbase} :
$$log|R_{k}(\alpha, \omega_{0},\ldots, \omega_{n})|=\nu log|\alpha|N\sum_{j=0}^{k-1}d^{j}+N\sum_{j=0}^{k-1}log|b(\alpha^{d^{j}})|+log|E_{N}(\alpha^{d^{k}})|, \text{ } \forall k\in\mathbb{N}.$$

Comme $b(z)$ est uniformément borné sur tout compact et que $b(0)\neq 0$, on a pour tout $k\in\mathbb{N}$ et pour tout $j\in\{0,\cdots, k-1\}$ :
\begin{equation}
\label{bineg}
- c_{9} \leq log|b(\alpha^{d^{j}})|\leq  c_{10}.
\end{equation}

En constatant que :
\begin{equation}
\label{somgeom}
\sum_{j=0}^{k-1} d^{j} = \frac{d^{k}-1}{d-1}\leq d^{k}, \text{ }\forall k\in\mathbb{N},
\end{equation}

et d'après \eqref{inegE} et \eqref{bineg}, il vient, pour tout $k\geq c_{0}(N)$ :

$$-c_{11}d^{k}N - c_{9}kN - c_{7}d^{k}n(N) \leq log|R_{k}(\alpha, \omega_{0},\ldots, \omega_{n})|\leq c_{10}kN -c_{8}d^{k}n(N).$$
En remarquant que :
$$kN=o_{k\rightarrow +\infty}(d^{k}n(N)),$$ 
et en utilisant \eqref{ord}, on obtient, quitte à augmenter $c_{0}(N)$, que pour tout $k\geq c_{0}(N)$:

\begin{equation}
\label{evalRk}
-c_{12}d^{k}n(N) \leq log|R_{k}(\alpha, \omega_{0},\ldots, \omega_{n})|\leq -c_{13}d^{k}n(N).
\end{equation}

Il nous reste à contrôler la hauteur de $R_{k}$. Pour ce faire, nous avons besoin d'étudier le degré en $z$ de $R_{k}$.\\
Notons $c_{14}$ le maximum des degrés du polynôme $a(z)$ et des polynômes de la matrice $M(z)$. Comme $R_{k-1}(z,X_{0},\cdots, X_{n})$ est homogène de degré $N$ en $X_{0},\cdots, X_{n}$, on obtient d'après \eqref{defrec} :
$$deg_{z}(R_{k})\leq c_{14}N + d\times deg_{z}(R_{k-1}), \text{ }\forall k\geq 1.$$
D'où par récurrence sur $k$ :

$$deg_{z}(R_{k})\leq c_{14}N\times \sum_{j=0}^{k-1} d^{j} +d^{k}deg_{z}(R_{0}), \text{ }\forall k\in\mathbb{N}.$$

D'après \eqref{somgeom} et comme $deg_{z}(R_{0})\leq N$, il vient :

\begin{equation}
\label{degzRk}
deg_{z}(R_{k})\leq c_{15}d^{k}N, \text{ }\forall k\in\mathbb{N}.
\end{equation}

On a par ailleurs besoin du lemme suivant.
\begin{lem}
\label{lemHRk}
Quitte à augmenter $c_{0}(N)$, on peut supposer que pour tout $k\geq c_{0}(N)$ :
\begin{equation}
\label{HRk}
h(R_{k})\leq c_{24}k^{3}.
\end{equation}
\end{lem}

\begin{proof}[Démonstration du lemme \ref{lemHRk}]
Pour tout $k\in\mathbb{N}$, posons :
$$R_{k}(z,X_{0},\cdots,X_{n})=\sum_{\underline{i}} a_{k,\underline{i}}z^{i_{n+1}}X_{0}^{i_{0}}X_{1}^{i_{1}}\cdots X_{n}^{i_{n}}, \text{ }a_{k,\underline{i}}\in\mathbb{K}.$$ 

Notons $\mathcal{E}\subset\mathbb{K}$ l'ensemble des coefficients du polynôme $a(z)$ et des polynômes de la matrice $M(z)$. On rappelle que l'on note $c_{14}$ le maximum des degrés du polynôme $a(z)$ et des polynômes de la matrice $M(z).$ 


On a d'après \eqref{defrec} :

\begin{equation}
\label{sommeU}
\sum_{\underline{i}}a_{k,\underline{i}}z^{i_{n+1}}X_{0}^{i_{0}}X_{1}^{i_{1}}\cdots X_{n}^{i_{n}} =\sum_{\underline{j}} a_{k-1,\underline{j}}z^{dj_{n+1}}(a(z)X_{0})^{j_{0}}(<M_{1}(z),\overline{X}>)^{j_{1}}\cdots(<M_{n}(z),\overline{X}>)^{j_{n}}.
\end{equation}

On pose :

$$a(z)=\sum_{s=0}^{c_{14}}a_{s}z^{s}, \text{ } a_{s}\in \mathbb{K},$$
$$M_{i}(z)=\left(\sum_{l=0}^{c_{14}}m_{i,j,l}z^{l}\right)_{1\leq j\leq n}, \text{ } m_{i,j,l}\in \mathbb{K}, i\in\{1,\cdots, n\},$$

\begin{align}
U_{k-1,\underline{j}}(z,X_{0},\cdots,X_{n}) &=a_{k-1,\underline{j}}z^{dj_{n+1}}(a(z)X_{0})^{j_{0}}(<M_{1}(z),\overline{X}>)^{j_{1}}\cdots(<M_{n}(z),\overline{X}>)^{j_{n}}\nonumber\\
&=a_{k-1,\underline{j}}z^{dj_{n+1}}\left(\sum_{s=0}^{c_{14}}a_{s}z^{s}X_{0}\right)^{j_{0}}\prod_{i=1}^{n}
\left(\sum_{j=1}^{n}\sum_{l=0}^{c_{14}}m_{i,j,l}z^{l}X_{j}\right)^{j_{i}}.
\label{U}
\end{align}

En développant \eqref{U} terme à terme, on obtient une somme d'au plus 
$$(c_{14}+1)^{j_{0}}((c_{14}+1)n)^{j_{1}+\cdots+j_{n}}\leq c_{16}^{N}$$
monômes en les $z,X_{0},\cdots, X_{n}$.

Par ailleurs, d'après \eqref{degzRk} et comme $R_{k-1}(z,X_{0},\cdots,X_{n})$ est homogène de degré $N$ en $X_{0},\cdots, X_{n}$, le polynôme $R_{k-1}(z,X_{0},\cdots,X_{n})$ possède au plus 
$$\begin{pmatrix}
N+n \\ n
\end{pmatrix}(c_{15}d^{k-1}N+1)\leq c_{17}(N)c_{18}^{k}$$
monômes en les $z,X_{0},\cdots, X_{n}$. Donc le développement terme à terme du membre de droite de \eqref{sommeU} produit une somme d'au plus 
$$c_{16}^{N}\times c_{17}(N)c_{18}^{k}\leq c_{19}(N)c_{18}^{k}$$
monômes en les $z,X_{0},\cdots, X_{n}$.

Donc chaque élément $a_{k,\underline{i}}z^{i_{n+1}}X_{0}^{i_{0}}X_{1}^{i_{1}}\cdots X_{n}^{i_{n}}$ du membre de gauche de \eqref{sommeU} est la somme d'au plus $c_{19}(N)c_{18}^{k}$ monômes du type 
$$uz^{i_{n+1}}X_{0}^{i_{0}}X_{1}^{i_{1}}\cdots X_{n}^{i_{n}}, u\in\mathbb{K},$$
fournis par le développement terme à terme du membre de droite de \eqref{sommeU}.
De plus, chaque coefficient $u$ de tels monômes est un produit de $N$ coefficients de $\mathcal{E}$ par un coefficient de $R_{k-1}(z,X_{0},\cdots, X_{n})$.

Par conséquent, pour toute place $w$ de $\mathbb{K},$ on peut écrire :
\begin{align}
max_{\underline{i}} |a_{k,\underline{i}}|_{w} &\leq c_{19}(N)c_{18}^{k} (max_{e\in \mathcal{E}} |e|_{w})^{N} \times max_{\underline{i}} |a_{k-1,\underline{i}}|_{w} \nonumber\\
& \leq c_{19}(N)c_{18}^{k} (max_{e\in \mathcal{E}} (1,|e|_{w}))^{N} \times max_{\underline{i}} (1,|a_{k-1,\underline{i}}|_{w}). \nonumber\\
\intertext{Donc}
log(max_{\underline{i}} (1,|a_{k,\underline{i}}|_{w})) & \leq log(c_{19}(N))+klog(c_{18}) + N log(max_{e\in \mathcal{E}} (1,|e|_{w})) \nonumber\\ 
\label{calc}
& \qquad + log(max_{\underline{i}} (1,|a_{k-1,\underline{i}}|_{w})).
\end{align}

On somme à présent \eqref{calc} sur toutes les places de $\mathbb{K}$. Remarquons tout d'abord qu'il n'y a en réalité qu'un nombre fini d'entre elles impliqué.

\begin{de}
\label{AS}
Soit $S$ un ensemble fini de places de $\mathbb{K}$ contenant l'ensemble des places archimédiennes de $\mathbb{K}$. On définit l'anneau des $S$-entiers, noté $A_{S}$, de la façon suivante :
$$A_{S}=\{x\in \mathbb{K}, |x|_{w}\leq 1, \forall w\notin S\}.$$
\end{de}

\begin{rem}
\label{placefinie}
\em{\label{somw}
Il existe un ensemble fini $S$ de places de $\mathbb{K}$ tel que :
$$\forall k,\forall \underline{i}, a_{k,\underline{i}}\in A_{S}.$$
En effet, il existe un ensemble fini $S$ de places de $\mathbb{K}$ tel que $A_{S}$ contienne l'ensemble des coefficients de $R_{0}(z,X_{0},\cdots, X_{n})$, de $a(z)$ et de ceux des polynômes de la matrice $M(z)$ puisque ceux-ci sont en nombre fini. D'après \eqref{defrec}, on voit par récurrence sur $k$ que pour tout $k\in\mathbb{N}$, chaque coefficient de $R_{k}$ est une somme de produits de coefficients des polynômes $R_{0}(z,X_{0},\cdots, X_{n})$, $a(z)$ et des polynômes de la matrice $M(z)$. La remarque \ref{placefinie} découle du fait que $A_{S}$ est un anneau. Nous noterons $c_{20}(N)$ le cardinal de $S$ (qui ne dépend pas de $k$) dans toute la suite.}
\end{rem}

En sommant \eqref{calc} sur toutes les places de $\mathbb{K}$ et en sachant qu'il y en a au plus $c_{20}(N)$ impliquées, on obtient :

\begin{align*}
h(R_{k}) &\leq c_{21}c_{20}(N)(log(c_{19}(N))+klog(c_{18})) + N c_{22} + h(R_{k-1})\\
&\leq c_{23}k^{2}+h(R_{k-1}), \text{ } \text{pour } k \text{ assez grand par rapport à } N.
\end{align*}
Puis par récurrence sur $k$ on obtient, quitte à augmenter $c_{0}(N)$, pour tout $k\geq c_{0}(N)$ :
$$h(R_{k})\leq c_{23}k^{3} + h(R_{0})\leq c_{24}k^{3}.$$
\end{proof}

\item
Elimination de la variable $z$
\medskip

Posons pour tout $k\in\mathbb{N}$ :
\begin{equation}
\label{defPk}
P_{k}(X_{0},\ldots, X_{n})=R_{k}(\alpha, X_{0},\ldots, X_{n})\in \mathbb{K}[X_{0},\ldots, X_{n}], \text{ }\forall k\in\mathbb{N}. 
\end{equation}

D'une part :
$$deg(P_{k})=deg_{X}(R_{k}) = N, \text{ } \forall k\in\mathbb{N}.$$
D'autre part, nous aurons besoin du lemme suivant.
\begin{lem}
\label{lemHPk}
Quitte à augmenter $c_{0}(N),$ on peut supposer que pour tout $k\geq c_{0}(N)$ :
$$h(P_{k})\leq c_{28}d^{k}N.$$

\end{lem}

\begin{proof}[Démonstration du lemme \ref{lemHPk}]
Pour tout $k\in\mathbb{N}$, notons $b_{k,\underline{j}}\in\mathbb{K}$ les coefficients du polynôme $P_{k}(X_{0},\cdots, X_{n})$ et rappelons que l'on note $a_{k,\underline{i}}\in\mathbb{K}$ les coefficients des polynômes $R_{k}(z,X_{0},\cdots, X_{n})$. D'après \eqref{defPk}, on voit qu'un coefficient de $P_{k}(X_{0},\cdots, X_{n})$ est une somme de produits d'une puissance de $\alpha$ par un coefficient de $R_{k}.$ D'après \eqref{degzRk}, le nombre de termes de cette somme est majoré par 
$$c_{15}d^{k}N + 1 \leq c_{25}d^{k}N.$$ 
Donc pour toute place $w$ de $\mathbb{K},$ on obtient :
\begin{align}
max _{\underline{j}}|b_{k,\underline{j}}|_{w} &\leq c_{25}d^{k}N \times max_{0\leq j\leq deg_{z}(R_{k})} |\alpha|_{w}^{j} \times max_{\underline{i}}|a_{k,\underline{i}}|_{w}\nonumber\\
&\leq c_{25}d^{k}N \times max_{0\leq j\leq deg_{z}(R_{k})} (1,|\alpha|_{w}^{j})\nonumber \\
& \qquad \times max_{\underline{i}}(1,|a_{k,\underline{i}}|_{w}) \nonumber\\
&\leq c_{25}d^{k}N \times (max(1,|\alpha|_{w}))^{deg_{z}(R_{k})} max_{\underline{i}}(1,|a_{k,\underline{i}}|_{w}). \nonumber
\intertext{Ainsi, pour tout $k$ assez grand par rapport à $N$, il vient :}
log(max_{\underline{j}}(1,|b_{k,\underline{j}}|_{w})) &\leq c_{26}k + deg_{z}(R_{k})log(max (1,|\alpha|_{w}))\nonumber\\
\label{b}
& \qquad + log(max_{\underline{i}}(1,|a_{k,\underline{i}}|_{w})).
\end{align}

On somme à présent \eqref{b} sur toutes les places de $\mathbb{K}$. Or, comme remarqué précédemment, chaque coefficient de $P_{k}(X_{0},\cdots, X_{n})$ est une somme de produits d'une puissance de $\alpha$ par un coefficient de $R_{k}(z,X_{0},\cdots, X_{n})$. Quitte à augmenter l'ensemble S de la remarque \ref{somw}, on peut supposer que $\alpha\in A_{S}$, ce qui implique que :
$$\forall k,\forall \underline{j}, b_{k,\underline{j}}\in A_{S}.$$

En sommant \eqref{b} sur les au plus $c_{20}(N)=card(S)$ places de $\mathbb{K}$ impliquées et quitte à augmenter $c_{0}(N)$, on obtient pour tout $k\geq c_{0}(N)$ :
 
\begin{align*}
h(P_{k}) &\leq c_{21}c_{20}(N) \times c_{26}k + c_{15}d^{k}N h(\alpha)  + h(R_{k}) \text{ d'après \eqref{degzRk}}\\
&\leq c_{21}c_{20}(N) \times c_{26}k + c_{15}d^{k}N h(\alpha)  + c_{24} k^{3} \text{ d'après \eqref{HRk}}\\
&\leq c_{27}k^{3} + c_{15}d^{k}N h(\alpha).
\end{align*}
En constatant que :
$$k^{3}=o_{k \rightarrow +\infty}(d^{k}N),$$
et quitte à augmenter $c_{0}(N),$ il vient pour tout $k\geq c_{0}(N) :$
$$h(P_{k})\leq c_{28}d^{k}N.$$
\end{proof}

Afin d'achever la démonstration du théorème \ref{nish_gen}, on remarque que, d'après \eqref{evalRk} et \eqref{defPk}, on a pour tout $k\geq c_{0}(N)$ :
$$-c_{12}d^{k}n(N) \leq log|P_{k}(\omega_{0},\ldots, \omega_{n})|\leq -c_{13}d^{k}n(N).
$$
On conclut la démonstration du théorème \ref{nish_gen} en appliquant le corollaire \ref{cas}, avec $s=l'$.

\end{enumerate}
\end{proof}

\begin{rem}
\em{
Cette démonstration contient quatre étapes classiques dans les preuves de l'indépendance algébrique de nombres $\omega_{i}$.
\smallskip

(i) Construction, pour chaque entier $N$, d'une fonction analytique de la variable $z$, dite fonction auxilliaire, ayant un grand ordre $n(N)$ (minoré en fonction de $N$) en $z=0$.
\medskip

(ii) Vérification du fait que cette fonction est non identiquement nulle.
\medskip

\noindent Pour chaque entier naturel $N$ fixé, (i) et (ii) permettent d'obtenir par récurrence (ici, itération $k$ fois de l'application $z\mapsto z^{d}$) des polynômes non nuls $\{P_{N,k}(X_{0},\cdots, X_{n})\}_{k}$ prenant de petites valeurs en $X_{i}=\omega_{i}$.
\medskip

(iii) et (iv) Pour chaque $N$ fixé, choix d'un polynôme $P_{N}(X_{0},\cdots, X_{n})=P_{N,k(N)}(X_{0},\cdots, X_{n})$ parmi les $\{P_{N,k}(X_{0},\cdots, X_{n})\}_{k}$ et majoration et minoration du nombre $|P_{N}(\omega_{0},\cdots,\omega_{n})|$ à l'aide de méthodes analytiques. L'ordre $n(N)$ apparaît naturellement dans la majoration et la minoration.
\medskip

On applique ensuite un critère d'indépendance algébrique du type du théorème \ref{philpoly} dans lequel on doit notamment vérifier la condition 5. La quantité $n(N)$ apparaissant naturellement dans (iii) et (iv), cette condition s'incarne, dans l'idée, en posant $t=N$ et $\epsilon(t)=\rho(t)=n(t)$. Or, il s'avère en général difficile d'étudier le comportement du quotient $n(N)/n(N+1)$ en fonction de $N$. Il est alors nécessaire de faire appel à un type de résultat souvent délicat à démontrer, appelé \textit{lemme de multiplicité}, qui fournit une majoration de $n(N)$ en fonction de $N$. Croisée avec la minoration de $n(N)$ considérée en (i), celle-ci permet de vérifier la condition 5 du théorème \ref{philpoly}. On retrouve par exemple ce schéma de démonstration dans \cite{B}.
\medskip

Ici, pour chaque $N$ fixé, on sait en fait majorer et minorer tous les nombres $\{|P_{N,k}(\omega_{0},\cdots,\omega_{n})|\}_{k}$, et pas simplement l'un d'entre eux, en (iii) et (iv). Pour chaque $N$ fixé, l'indice $k$ est donc libre de varier. Cela nous permet de poser $t=k$ et, dans l'idée, $\epsilon(t)=\rho(t)=n(N)d^{t}$ et de rendre le quotient $\epsilon(t)/\epsilon(t+1)$ facile à contrôler par rapport à $t$. Cela permet de vérifier directement la condition 5 du théorème \ref{philpoly} sans avoir recours à un \textit{lemme de multiplicité} qui compliquerait la démonstration. Cette particularité constitue un avantage notable de la méthode de Mahler.}

\end{rem}

\section{Analogue du théorème de Mahler pour les corps de fonctions en caractéristique non nulle}
\label{mahler}

Dans cette section, nous nous intéressons à des fonctions $f(z)\in\mathbb{K}[[z]]$ pour lesquelles il existe un entier $d\geq 2$ et deux polynômes $A(z,X), B(z,X)\in \mathbb{K}[z, X]$, $B(z,X)$ $\cancel\equiv$  $0$, tels que :
\begin{equation}
\label{defmahlrat}
f(z^{d})=\frac{A(z,f(z))}{B(z,f(z))}.
\end{equation}

On fixe pour la suite les notations suivantes :
$$A(z,X)=\sum_{i=0}^{m} a_{i}(z)X^{i}, B(z,X)=\sum_{i=0}^{m} b_{i}(z)X^{i},$$
où
$$a_{i}(z), b_{i}(z)\in \mathbb{K}[z], \forall 1\leq i\leq m \text{  et } m=max\{deg_{X}(A), deg_{X}(B)\}. $$
Enfin, $\Delta(z)\in\mathbb{K}[z]$ désignera le résultant des polynômes $A$ et $B$ par rapport à la variable $X$.

Pour ce type d'équation fonctionnelle, la définition d'un point régulier est la suivante.
\begin{de}
On dira qu'un nombre $\alpha\in C$ est régulier pour l'équation \eqref{defmahlrat} si pour tout $k\in\mathbb{N}$, $\alpha^{d^{k}}$ n'est pas un zéro de $\Delta(z)$.
\end{de}

Nous énonçons ci-dessous le théorème que nous souhaitons établir dans cette section. Celui-ci est valable pour tout choix préalable de l'un ou l'autre des deux contextes de la section \ref{contextes}. 

\begin{thm}
\label{thmmahl}
Soit $\mathbb{K}$ une extension finie de K. Soit $f(z)\in \mathbb{K}\{z\}$ une fonction analytique dans un voisinage $\mathcal{V}$ de l'origine inclus dans le disque unité de $C$ et solution de l'équation \eqref{defmahlrat}. On suppose de plus que l'on a :
$$m<d.$$
Soit $\alpha\in\overline{\mathbb{K}}\cap \mathcal{V}$ un nombre non nul régulier pour l'équation \eqref{defmahlrat}. 
\medskip

Alors, l'égalité suivante est vérifiée :
\begin{equation}
degtr_{\mathbb{K}}\{f(\alpha)\}=degtr_{\mathbb{K}(z)}\{f(z)\}.
\end{equation}

Autrement dit, si la fonction $f(z)$ est transcendante sur $\mathbb{K}(z)$, alors le nombre $f(\alpha)$ est transcendant sur $\mathbb{K}$ (et réciproquement).
\end{thm}

Le cas où $\mathbb{K}$ est un corps de nombres est dû à Mahler \cite{M1} (voir aussi \cite[p.5]{N}). Nous nous proposons de démontrer ce théorème dans l'esprit de la démonstration du théorème \ref{nish_gen}. Celle-ci sera encore une fois valable aussi bien pour les corps de nombres que pour les corps de fonctions en caractéristique non nulle. La principale différence avec la démonstration du théorème \ref{nish_gen} provient du fait que notre étude ne porte plus sur l'indépendance algébrique mais seulement sur la transcendance. Nous n'aurons donc pas besoin d'utiliser un résultat aussi fort que le critère d'indépendance algébrique de Philippon. En remplacement, l'inégalité de Liouville (proposition \ref{liouville}, point 3), qui elle est élémentaire, suffira.

\begin{proof}
On suppose que la fonction $f(z)$ est transcendante sur $\mathbb{K}(z)$ (notons qu'alors $m>0$) et on veut montrer que le nombre $f(\alpha)$ est transcendant sur $\mathbb{K}$, la réciproque étant claire.
Quitte à remplacer $\mathbb{K}$ par $\mathbb{K}(\alpha)$, on supposera avoir $\alpha\in\mathbb{K}$.
Soit $N$ un entier naturel non nul fixé. Dans la suite de cette démonstration, pour tout $j\in\mathbb{N}$, le terme $c_{j}$ désignera une constante strictement positive ne dépendant que de $\mathbb{K},$ de $f(z)$ et de $\alpha$ et $c_{j}(N)$ désignera une constante strictement positive ne dépendant que de $\mathbb{K},$ de $f(z)$, de $\alpha$ et de $N$. Il existe un polynôme non nul $R(z,X)\in  \mathbb{K}[z,X]$ vérifiant les conditions suivantes :
\begin{enumerate}
\item
$deg_{z}(R)\leq N$
\item
$deg_{X}(R)\leq N$
\item
\begin{equation}
n(N):=ord_{z=0}R(z,f(z))\geq N^{2}
\end{equation}
\end{enumerate}
En effet, en regardant $R$ comme un polynôme en $X$ à coefficients des polynômes $P(z)$ de $\mathbb{K}[z]$, cela revient à résoudre un système de $N^{2}$ équations en les $(N+1)^{2}$ inconnues que sont les coefficients des polynômes $P(z)$, équations à coefficients dans $\mathbb{K}$. Puisque $N^{2}<(N+1)^{2}$, il en existe bien une solution non triviale à coefficients dans $ \mathbb{K}.$\\ 
Notons :
$$E_{N}(z)=R(z,f(z))=\sum_{j=n(N)}^{+\infty}a_{j}(N)z^{j}.$$
Comme $f(z)$ est transcendante, $E_{N}(z)$ $\cancel\equiv$  $0$.
\medskip

Posons $R_{0}(z,X)=R(z,X)$ et pour tout $k\geq 1 :$
\begin{equation}
\label{rec_origin}
R_{k}(z,X)=R_{k-1}\left(z^{d},\frac{A(z,X)}{B(z,X)}\right)B(z,X)^{m^{k-1}N}.
\end{equation}

On montre par récurrence sur $k$ que pour tout $k\in\mathbb{N}$, $R_{k}(z,X)\in\mathbb{K}[z,X]$ avec :
\begin{equation}
\label{degx}
deg_{X}(R_{k})\leq m^{k}N,
\end{equation}
et :
\begin{equation}
\label{prodmahl}
R_{k}(z,f(z))=\prod_{j=0}^{k-1} B\left(z^{d^{j}},f(z^{d^{j}})\right)^{m^{k-1-j}N}E_{N}(z^{d^{k}}).
\end{equation}

D'autre part, un calcul rapide montre que, pour tout $k$ assez grand par rapport à $N$, disons $k\geq c_{0}(N)$, on a :  
\begin{equation}
\label{ENdiff0mahl}
E_{N}(\alpha^{d^{k}})\neq 0,
\end{equation}

et :

\begin{equation}
\label{maj}
-c_{1}d^{k}n(N)\leq log|E_{N}(\alpha^{d^{k}})|\leq -c_{2}d^{k}n(N).
\end{equation}

Par ailleurs, l'équation \eqref{defmahlrat} entraîne que pour tout $k\geq 1$ :
$$B(\alpha^{d^{k-1}},f(\alpha^{d^{k-1}}))f(\alpha^{d^{k}})=A(\alpha^{d^{k-1}},f(\alpha^{d^{k-1}})),$$
Comme $\Delta(\alpha^{d^{k-1}})\neq 0$ par hypothèse, il vient :
\begin{equation}
\label{Balpha}
B(\alpha^{d^{k-1}},f(\alpha^{d^{k-1}}))\neq 0, \text{ } \forall k\geq 1.
\end{equation}

D'après \eqref{prodmahl}, \eqref{ENdiff0mahl} et \eqref{Balpha}, on trouve pour tout $k\geq c_{0}(N)$ :
\begin{equation}
\label{Rknonnul}
R_{k}(\alpha,f(\alpha))\neq 0.
\end{equation}

L'égalité \eqref{prodmahl} est l'analogue de \eqref{egrecbase} de la démonstration du théorème \ref{nish_gen} et on en déduit de la même façon à l'aide de \eqref{maj} que, quitte à augmenter $c_{0}(N)$ :
\begin{equation}
\label{valeurRk}
log|R_{k}(\alpha,f(\alpha))|\leq -c_{3}d^{k}n(N), \text{ } \forall k\geq c_{0}(N).
\end{equation}

Posons pour tout $k\in\mathbb{N}$ :
\begin{equation}
\label{defPk_origin}
P_{k}(X)=R_{k}(\alpha,X)\in\mathbb{K}[X].
\end{equation}
On a d'après \eqref{Rknonnul} et \eqref{valeurRk}, pour tout $k\geq c_{0}(N)$ :
\begin{equation}
\label{Pknonnul}
P_{k}(f(\alpha))\neq 0,
\end{equation}

et :

$$log|P_{k}(f(\alpha))|\leq -c_{3}d^{k}n(N).$$

Supposons par l'absurde que $f(\alpha)$ est algébrique sur $\mathbb{K}$. Quitte à remplacer $\mathbb{K}$ par $\mathbb{K}(f(\alpha))$, on peut dire que $f(\alpha)\in\mathbb{K}$. Cela et \eqref{Pknonnul} entraînent que $P_{k}(f(\alpha))\in\mathbb{K}^{*}$ pour tout $k\geq c_{0}(N)$. L'inégalité de Liouville fournit alors :

\begin{equation}
\label{liouH}
log|P_{k}(f(\alpha))|\geq -c_{4} h(P_{k}(f(\alpha))), \text{ } \forall k\geq c_{0}(N).
\end{equation}

Estimons à présent la hauteur $h(P_{k}(f(\alpha)))$. Etudions tout d'abord les quantités $deg_{z}(R_{k})$ et $h(R_{k})$.
\medskip

Par récurrence sur $k$, on obtient :

$$deg_{z}(R_{k})\leq c_{5}N\sum_{j=0}^{k-1} m^{j}d^{k-1-j} +d^{k}deg_{z}(R_{0}), \text{ }\forall k\in\mathbb{N},$$
où 
$$ c_{5}=max \{deg_{z}(A), deg_{z}(B)\}.$$

En remarquant que, pour tout $k\in\mathbb{N}$ :
\begin{align*}
\sum_{j=0}^{k-1} m^{j}d^{k-1-j} &=d^{k-1}\left(1+\frac{m}{d}+\cdots+\left(\frac{m}{d}\right)^{k-1}\right)\\
&\leq c_{6}d^{k-1}, \text{ car } m<d, 
\end{align*}

il vient :
\begin{equation}
\label{degz}
deg_{z}(R_{k}) \leq c_{7}d^{k-1}N+d^{k}N\leq c_{8}d^{k}N, \text{ } \forall k\in\mathbb{N}.
\end{equation}

De manière analogue à la démonstration du lemme \ref{lemHRk} dans la preuve du théorème \ref{nish_gen}, on peut montrer que l'on a ici, quitte à augmenter $c_{0}(N)$ :
\begin{equation}
\label{hmahl}
h(R_{k})\leq c_{9}(N)m^{k}+k^{3}, \text{ } \forall k\geq c_{0}(N).
\end{equation}

Posons pour tout $k\in\mathbb{N}$ :
\begin{equation}
\label{Pk_origin}
P_{k}(f(\alpha))=R_{k}(\alpha,f(\alpha))=\sum_{i\leq deg_{z}(R_{k}), j\leq deg_{X}(R_{k})} d_{k,i,j}\alpha^{i}f(\alpha)^{j}, \text{ } d_{k,i,j}\in\mathbb{K}.
\end{equation}

D'après \eqref{degx} et \eqref{degz}, pour toute place $w$ de $\mathbb{K}$, on peut écrire :
$$|P_{k}(f(\alpha))|_{w}\leq (c_{8}d^{k}N+1)(m^{k}N+1)\times \max_{i,j}|d_{k,i,j}|_{w} \times (max\{1,|\alpha|_{w}\})^{c_{8}d^{k}N}\times (max\{1,|f(\alpha)|_{w}\})^{m^{k}N}.$$

D'où :
\begin{align}
\label{calcalpha}
log(max\{1,|P_{k}(f(\alpha))|_{w}\}) &\leq c_{10}(N)+c_{11}k + log(\max_{i,j}\{1,|d_{k,i,j}|_{w}\}) + c_{8}d^{k}Nlog(max\{1,|\alpha|_{w}\})\nonumber \\
& \qquad + m^{k}Nlog(max\{1,|f(\alpha)|_{w}\}).
\end{align}

Comme $\alpha,f(\alpha)\in\mathbb{K}$, il existe un ensemble fini $S$ de places de $\mathbb{K}$ tel que $\alpha, f(\alpha)\in A_{S}$ (voir la définition \ref{AS}). D'après \eqref{Pk_origin}, quitte à augmenter l'ensemble $S$, on obtient de la même manière que pour la remarque \ref{placefinie} : 
$$\forall k\in\mathbb{N}, P_{k}(f(\alpha))\in A_{S}.$$

En posant $card(S)=c_{12}(N)$ et en sommant \eqref{calcalpha} sur les au plus $c_{12}(N)$ places de $\mathbb{K}$ impliquées, on trouve, quitte à augmenter $c_{0}(N)$, pour tout $k\geq c_{0}(N)$ :

\begin{align*}
h(P_{k}(f(\alpha)))&\leq c_{13}(N)k + h(R_{k}) + c_{8}d^{k}Nh(\alpha) + m^{k}Nh(f(\alpha))\\
&\leq c_{13}(N)k + c_{9}(N)m^{k} + k^{3} + c_{8}d^{k}Nh(\alpha) + m^{k}Nh(f(\alpha)), \text{ d'après \eqref{hmahl}}\\
&\leq c_{14}d^{k}N, \text { car } m<d.
\end{align*}

Ainsi, \eqref{liouH} devient :
\begin{equation}
\label{liou_appli}
log|P_{k}(f(\alpha))|\geq -c_{15}d^{k}N, \text{ } \forall k\geq c_{0}(N).
\end{equation}

D'après \eqref{valeurRk} et \eqref{liou_appli}, on obtient :
$$-c_{15}d^{k}N\leq log|P_{k}(f(\alpha))|\leq -c_{3}d^{k}n(N), \text{ } \forall k\geq c_{0}(N).$$
D'où :
$$-c_{15}d^{k}N\leq -c_{3}d^{k}n(N), \text{ } \forall k\geq c_{0}(N).$$

En divisant par $d^{k}$ les deux côtés de l'inégalité précédente, on obtient :
$$-c_{15}N\leq -c_{3}n(N), \text{ } \forall k\geq c_{0}(N).$$
Or, on a : 
$$n(N)\geq N^{2}.$$
En choisissant au début de cette preuve $N$ assez grand, on aboutit à une contradiction.
\medskip

Donc $f(\alpha)$ est transcendant, ce qui conclut la démonstration du théorème \ref{thmmahl}.
\end{proof}

Nous donnons à présent un exemple d'application à ce théorème. Soit $n\in\mathbb{N}$ tel que $p\nmid n$. Cela garantit que $\frac{1}{n^{k}}\in\mathbb{Z}_{p}$, pour tout entier naturel $k$. Le théorème de Lucas \cite{Lucas} (voir aussi \cite{AMFvdP}) permet alors de définir la fonction suivante :
\begin{equation}
\label{exmahl}
f(z)=\prod_{k=0}^{\infty} \left(1-Tz^{q^{k}}\right)^{\frac{1}{n^{k}}}.
\end{equation}
Notons que la fonction $f(z)$ est analytique sur $\{|z|<\frac{1}{q}\}$ et que $f(z)\in\mathbb{F}_{q}(T)[[z]]$. On a alors le théorème suivant.

\begin{thm} 
Si $n<q$, alors pour tout nombre $\alpha\in \overline{K}^{*}\cap \{|z|<\frac{1}{q}\}$, $f(\alpha)$ est un nombre transcendant sur $\mathbb{F}_{q}(T)$.
\end{thm}

\begin{proof}

La fonction $f(z)$ satisfait à l'équation suivante :
\begin{equation}
\label{exeq}
f(z^{q})=\frac{f(z)^{n}}{(1-Tz)^{n}}.
\end{equation}

Par ailleurs, on remarque que la fonction $f(z)$ a une infinité de zéros sur $\overline{K}$, ce qui implique qu'elle est transcendante sur $K(z)$.
De plus, on observe que :
$$\Delta(z)=(1-Tz)^{n}.$$
Donc les seules singularités de l'équation \eqref{exeq} sont les racines $q^{k}$-ièmes de $\frac{1}{T}$, pour tout entier naturel $k$, dont aucune n'appartient à l'ensemble $\{|z|<\frac{1}{q}\}$.
Ainsi, comme $n<q$ par hypothèse, on peut appliquer le théorème \ref{thmmahl} avec :
$$d=q,\text{ } m=n,\text{ } A(z,X)=X^{n}, \text{ } B(z,X)=(1-Tz)^{n},$$

et $\alpha\in \overline{K}^{*}\cap \{|z|<\frac{1}{q}\}$ pour conclure que le nombre $f(\alpha)$ est transcendant sur $K$. 

\end{proof}

\section{Perspectives}
Si l'on considère un vecteur solution $(f_{1}(z),\cdots, f_{n}(z))$ du système mahlérien \eqref{syst_gen} et que l'on suppose la fonction $f_{1}(z)$ transcendante sur $\mathbb{K}(z)$, le théorème \ref{nish_gen} ne permet pas de conclure que le nombre $f_{1}(\alpha)$ est transcendant sur $\mathbb{K}$, où $\alpha$ est un nombre algébrique non nul régulier pour le système \eqref{syst_gen} (mais seulement qu'il existe un nombre transcendant parmi les $f_{i}(\alpha)$). Ce type d'implication nécessite un raffinement du théorème \ref{nish_gen} garantissant que s'il existe un polynôme homogène $P(X_{1},\cdots, X_{n})\in\mathbb{K}[X_{1},\cdots, X_{n}]$ tel que $P(f_{1}(\alpha),\cdots, f_{n}(\alpha))=0$, alors il existe un polynôme $Q(z,X_{1},\cdots, X_{n})\in\mathbb{K}[z,X_{1},\cdots, X_{n}]$ homogène en $X_{1},\cdots, X_{n}$ tel que $Q(z,f_{1}(z),\cdots, f_{n}(z))=0$ et $Q(\alpha,X_{1},\cdots, X_{n})=P(X_{1},\cdots, X_{n})$. Autrement dit, il serait intéressant d'établir que, sous les hypothèses du théorème \ref{nish_gen}, toute relation algébrique entre les nombres est obtenue comme spécialisation d'une relation algébrique entre les fonctions.
Cette précision a été apportée au théorème de Siegel-Shidlovskii par Beukers \cite{Beukers} en 2006. Pour les systèmes aux $\sigma$-différences, ce résultat est dû à Anderson, Brownawell et Papanikolas \cite{ABP1}. Une telle description dans le cadre du théorème de Ku. Nishioka pour les corps de nombres a été obtenue par Philippon \cite{Phil-S-S} en 2015 dans le cas de relations algébriques inhomogènes et parachevée par Adamczewski et Faverjon \cite {A-F} la même année. 
\bigskip

\textbf{Remerciements}. L'autrice tient à remercier Patrice Philippon pour lui avoir confirmé que son critère d'indépendance algébrique s'appliquait au cadre de cet article et pour avoir répondu avec rapidité, précision et bienveillance à ses innombrables questions à ce sujet; Federico Pellarin pour les pistes de réflexions et les éclaircissements qu'il lui a apporté.e.s autour des corps de fonctions en caractéristique non nulle; Evgeniy Zorin pour toutes les discussions enthousiastes et enrichissantes sur la théorie de l'élimination notamment et, évidemment, Boris Adamczewski pour ses nombreuses et patientes relectures et corrections de ce texte. Il est entendu que toute critique ne saurait être imputée qu'à la seule autrice de cet article.




\bibliographystyle{smfplain}
\bibliography{nishioka_en_car_p_texmaker}
\printindex

\nocite{P2}
\nocite{M2}
\nocite{M3}
\nocite{ABP2}
\nocite{P_elim}

\address

\end{document}